\documentclass[12pt,reqno]{article}
\usepackage{times}
\usepackage[utf8]{inputenc}
\usepackage{hyperref}
\usepackage{physics}
 
\usepackage{amsmath}
\usepackage{amsthm}
\usepackage{amsfonts}
\usepackage{amssymb}
\usepackage{epic}
\usepackage{eepic,epsfig}
\usepackage{times,bm,macros}
\usepackage{graphics}
\usepackage{epsfig,multicol,bbm}
\usepackage{subcaption}
\usepackage{bbm}

\usepackage{xcolor}
\usepackage{graphicx}

\theoremstyle{plain}
\newtheorem{thm}{Theorem}

\newtheorem{propn}{Proposition}
\newtheorem{lem}{Lemma}

\newtheorem{defn}{Definition}

\theoremstyle{remark}
\newtheorem{rem}{Remark}

\newcommand{\ra}{\to}

\def\beq{\begin{equation}}
\def\ee{\end{equation}}
\def\bea{\begin{eqnarray}}
\def\eea{\end{eqnarray}}

\newcommand{\al}{\alpha}

\newcommand{\Ga}{\Gamma}
\newcommand{\de}{\delta}

\newcommand{\ep}{\epsilon}
\newcommand{\la}{\lambda}

\newcommand{\pa}{\partial}

\newcommand{\CD}{{\mathcal D}}

\newcommand{\CH}{{\mathcal H}}
\newcommand{\CI}{{\mathcal I}}

\newcommand{\CL}{{\mathcal L}}

\newcommand{\CO}{{\mathcal O}}

\newcommand{\CS}{{\mathcal S}}

\newcommand{\SH}{{\mathsf H}}

\newcommand{\SQ}{{\mathsf Q}}  
\newcommand{\SR}{{\mathsf R}}
\renewcommand{\SS}{{\mathsf S}}

\newcommand{\BR}{{\mathbb R}}

\newcommand{\BC}{{\mathbb C}}

\newcommand{\BP}{{\mathbb P}}

\newcommand{\BZ}{{\mathbb Z}}

\newcommand{\rf}[1]{(\ref{#1})}

\begin{document}\thispagestyle{empty}
\title{Analytic Langlands Correspondence from Separation of Variables}
\author{Federico Ambrosino$^{1,\dagger}$, J\"org Teschner$^{1,2,*}$}
\address{
$^{1}$  Deutsches Elektronen-Synchrotron DESY, \\
Notkestrasse 85,
20607 Hamburg,
Germany,\\[2ex]
$^{2}$ Department of Mathematics, 
University of Hamburg, \\
Bundesstrasse 55,
20146 Hamburg, Germany,\\[2ex]
E-mail: ${}^\dagger$ \href{mailto:federicoambrosino25@gmail.com}{federicoambrosino25@gmail.com}, \; $ {}^*$\href{mailto:joerg.teschner@desy.de}{joerg.teschner@desy.de}}
\maketitle

\begin{quote}
\begin{center}{\bf Abstract}\end{center}
{\small
The analytic Langlands correspondence proposed by Etingof, Frenkel and Kazhdan
describes the solution to the spectral problems naturally arising in the quantisation of the Hitchin
integrable systems in terms of real opers, certain second order differential operators on a Riemann
surface having real monodromy. We prove this correspondence  in the cases associated to the
group $\mathrm{PSL}(2,\mathbb{C})$, and Riemann surfaces of genus zero with a number of
punctures larger than three. A crucial ingredient is a unitary integral transformation mapping
products of solutions to the ordinary differential equation associated to a real oper to eigenfunctions
of the quantised Hitchin Hamiltonians. This allows us to construct joint eigenfunctions of Hecke
operators and Hitchin Hamiltonians from real opers.

}
\end{quote}

\setcounter{tocdepth}{2}
\tableofcontents



\section{Introduction}

The analytic Langlands correspondence has been introduced in \cite{EFK1}. It supplements the geometric Langlands correspondence
by analytic aspects, claiming a correspondence between square-integrable eigenfunctions of the Hamiltonians of the quantised Hitchin 
integrable system on a Riemann surface $C$, and real opers, a certain class of flat connections on $C$
having real holonomy. It refines the correspondence 
which had previously been
proposed in \cite{Te} between single-valued eigenfunctions of the Hamiltonians of the Hitchin integrable system, and real opers. 
Such correspondences have interesting relations  outlined in \cite{GW,GT}
to supersymmetric quantum field theories, and to two-dimensional conformal field
theories.

The proposal made in \cite{EFK1} has  two main ingredients helping to gain analytic control. 
The first is to work in the Hilbert space framework offered by the natural scalar product on spaces of half-densities on 
the moduli spaces $\mathrm{Bun}_{G}(C)$ of holomorphic $G$-bundles on $C$. Another crucial ingredient is the use of the Hecke operators,
natural families of integral operators on the spaces of half-densities on $\mathrm{Bun}_{G}(C)$ \cite{EFK2}. As eigenfunctions of the 
Hecke operators turn out to be simultaneously eigenfunctions of the Hitchin Hamiltonians, one may start by studying the spectral problem 
for the Hecke operators. This spectral problem is more easily seen to be well-defined, as the Hecke operators are compact
self-adjoint operators, while the Hamiltonians 
of the quantised Hitchin system are unbounded. 

A proof of the analytic Langlands correspondence so far only appears to be available in a few cases, associated to $G=\mathrm{SL}(2)$, 
and Riemann surfaces $C=C_{g,n}$ of genus $g=0$ and $n=4,5$ punctures \cite{EFK1,EFK}. The goal of this paper is to offer a proof
for arbitrary $n$.

To this aim we shall adopt a strategy known as Separation of Variables ($\mathsf{SoV}$) in the literature on quantum integrable models. It amounts to the 
construction of a unitary integral transformation mapping the product of solutions to an ordinary differential equation (ODE) to a solution of the 
eigenvalue equations of our interest. 
The ODEs appearing in this context are closely related to the real opers appearing on one
side of the analytic Langlands correspondence. The $\mathsf{SoV}$ method is powerful  as it allows one to construct
eigenfunctions of the Hitchin Hamiltonians from real opers. It shows that having a correspondence to a real oper is not 
only necessary, as shown in \cite{EFK,EFK4} for arbitrary $n$, but also sufficient for being an eigenfunction.

To carry this program out in practice, it will be crucial to understand how the Hecke operators interact with the integral transformations 
defined by the $\mathsf{SoV}$ method. This will allow us to prove that the image of a real oper under the $\mathsf{SoV}$-transformation is an eigenfunction 
of the Hecke operators.  It turns out that
the Hecke eigenvalue property follows from the observation that the result of a Hecke modification at a point $z\in C\equiv C_{0,n}$ can be 
represented by a $\mathsf{SoV}$-transformation associated to the Riemann surface $C':=C\setminus\{z\}\simeq C_{0,n+1}$.
Proving this fact turns out to be the main technical issue arising in this context.

The application of the $\mathsf{SoV}$  method circumvents a difficulty that appears in other approaches
to the  analytic Langlands correspondence. It originates from the fact that the Hitchin Hamiltonians are singular
at the locus in $\mathrm{Bun}_{G}(C)$ associated to the ``wobbly'' bundles, bundles admitting nilpotent Higgs fields. Checking that
the resulting singular behaviour of the eigenfunctions at the locus does not spoil square-integrability can be difficult, having to rely on detailed
information on the geometry of the wobbly locus. The $\mathsf{SoV}$ transformation, being unitary, maps the check of square-integrability 
to the simpler problem to check square-integrability of products of solutions to the oper equation with respect to a suitable measure. 

We expect that the strategy used in this paper admits a generalisation to Riemann surfaces $C=C_{g,n}$ of higher genus
$g>0$ using the classical version of the $\mathsf{SoV}$-transformation developed in \cite{DT} as geometric groundwork.

\section{Analytic Langlands correspondence}
We will study the quantised Hitchin systems on 
$C=\mathbb{P}^1\setminus\{z_1,\dots,z_{n}\}$. In order to formulate the main results it will 
be convenient to assume, without essential loss of generality, that $z_n=\infty$.
The relevant geometrical background is
described in \cite{EFK}.

\subsection{Quantisation of the Hitchin system}

We begin by  introducing Hilbert spaces $\CH_{J}$, of homogeneous and translation invariant 
functions. The elements of $\CH_J$ are
represented by functions $\Psi$ of $n-1$ complex variables $\bm{x}=(x_1,\dots,x_{n-1})$
satisfying
\begin{equation}\label{trans-dil-Psi}
\begin{aligned}
&\Psi(x_1+\de,\dots,x_{n-1}+\de)=\Psi(x_1,\dots,x_{n-1}), \\
&\Psi(\rho x_1,\dots,\rho x_{n-1})=|\rho|^{2J}\Psi(x_1,\dots,x_{n-1}).
\end{aligned}
\end{equation}
It is well-known that such functions can be used to represent invariants in tensor 
products of $n$ spherical principal series representations of $\mathrm{PSL}(2,\BC)$ if $j_r\in-\frac{1}{2}+\mathrm{i}\BR$, 
$r=1,\dots,n$.
Translation invariance allows one to represent  the functions $\Psi$ by
the homogenous functions 
$\Psi(0,x_2,\dots,x_{n-1})$. 
Away from $x_{n-1}=0$ one may represent such functions in the form
\begin{equation}\label{Psi-psi}
\Psi(0,x_2,\dots,x_{n-1})=|x_{n-1}|^{2J} \psi(x_2/x_{n-1},\dots,x_{n-2}/x_{n-1}).
\end{equation}
Similar representations exist away from $x_r=0$, $r=2,\dots,n-1$. It follows that 
the functions $\psi$  on 
the right side of \rf{Psi-psi} represent sections of line bundles $\CL_{J}$ of 
densities on the complex projective space $\BP^d$, $d=n-3$,  in a particular chart.

The Hilbert space structure on $L^2(\BC^d)$ allows us to define a 
natural norm on spaces of regular homogenous functions $\Psi$. With the help of  \rf{Psi-psi} we may represent this norm  as
\begin{equation}\label{Psi-norm}
\lVert \Psi\rVert^2=\int_{\BC^d} d^{2d}\bm{\xi}\;|\psi(\bm{\xi})|^2, \qquad \bm{\xi}=(\xi_2,\dots,\xi_{n-2}).
\end{equation}
Representatives for the scaling equivalence class can be
represented in many  other ways than \rf{Psi-psi}.  When $J\in -\frac{n-2}{2}+\mathrm{i}\BR$ it can be checked that $\lVert \Psi\rVert^2$ does not 
depend on the representation, 
explaining why \rf{Psi-norm} defines a natural Hilbert space $\CH_{J}$ of square-integrable sections of $\CL_J$.

Following \cite{EFK}, we shall use $\CH_J$ to define $\CH_{i}$ for $i\in\BZ/2$ as $\CH_i=\CH_{J_i}$, with $J_i$ chosen as
\begin{equation}\label{jidef}
J_i=\sum_{r=1}^{n-1}j_r-j_{n,i}, \qquad j_{n,i}=\bigg\{\begin{aligned} j_n\quad\;\;&\text{for}\;\;i=0,\\
-j_n-1\;\;&\text{for}\;\;i=1.
\end{aligned}
\end{equation}
We may note that the Hilbert spaces $\CH_i$, $i=0,1$, are canonically isomorphic,
with unitary isomorphisms $\mathbf{K}:\CH_i\ra\CH_{i+1}$ being represented in terms of 
the functions $\psi$ representing elements of $\CH_i$ via \rf{Psi-psi} as the identity operator, 
 but different representations in terms  of homogenous functions. As explained in \cite{EFK}, 
 one may associate $\CH_i$ with  the quantised Hitchin systems
 associated to  vector bundles on $C$ of the form $\mathcal{O}\oplus \mathcal{O}(i)$, for 
 $i=0,1$, respectively. 

The quantum integrable structure of the Hitchin system 
is defined by the Hamiltonians.
In order to define the Hamiltonians, we shall first introduce the differential operators
\begin{equation}
\CD_r^{+}=\pa_{x_r},\qquad \CD_{r}^0=x_r\pa_{x_r}-j_r,\qquad\CD_r^-=2j_rx_r-x_r^2\pa_{x_r},
\end{equation} 
allowing us to define
\[
\mathsf{H}_r=\sum_{s\neq r}\frac{\eta_{ab}\CD_r^a\CD_s^b}{z_r-z_s}, \qquad r=1,\dots,n-1,
\]
where $\eta^{00}=1$, $\eta^{+-}=\eta^{-+}=-\frac{1}{2}$, and $\eta^{ab}=0$ otherwise.
One should keep in mind that only $n-3$ out of the $n-1$ operators $\mathsf{H}_r$ are linearly independent. We are here considering 
parameters of the form $j_r=-\frac{1}{2}+\mathrm{i}s_r$ with $s_r\in\BR\setminus\{0\}$, corresponding to the cases 
of principal series representations of $\mathrm{PSL}(2,\BC)$ discussed in \cite{EFK4}, excluding the slightly different 
case $j_r=-\frac{1}{2}$ considered in \cite{EFK}. 

The differential operators $\SH_r$ will be unbounded, as usual. In order to control the functional-analytic aspects
it will be advantageous to introduce the so-called Baxter operators.

\subsection{Baxter operators}

We shall define the operators $\mathbf{Q}_y$ as the composition $\mathbf{Q}_y=\mathbf{K}\cdot \mathbf{H}_y$, with the 
{Hecke operators}
$\mathbf{H}_y:\CH_i\ra\CH_{i+1}$, being defined as 
\begin{equation}\label{Heckedef}
\begin{aligned}
&(\mathbb{H}_{y,x}\Psi)(x_1,\dots,x_{n-1})=\prod_{r=1}^{n-1}\frac{|x_r-x|^{4j_r}}{|z_r-z|^{2j_r}}
\Psi\big(y_1(x_1,x,y),\dots,y_{n-1}(x_{n-1},x,y)\big),\\
&\mathbf{H}_{y}\Psi=\frac{1}{\pi}\int_{\BC}d^2x \;\mathbb{H}_{y,x}\Psi,
\qquad\qquad y_r(x_r,x,y):=-\frac{z_r-y}{x_r-x},\qquad 
\end{aligned}
\end{equation}
We refer to the operators $\mathbf{Q}_y$ as Baxter operators\footnote{This terminology can be motivated by 
observing that the equation \rf{univoper} plays the role of Baxter's T-Q-relations in this context, see also the 
discussion in \cite{EFK4}. We note, however, that the precise use of this terminology  adopted here can differ slightly from other appearances in the literature.} in order to distinguish the operators  $\mathbf{Q}_y:\CH_i\ra\CH_{i}$
from their relatives $\mathbf{H}_y$ mapping $\CH_i$ to $\CH_{i+1}$.
The main properties of the operators $\mathbf{Q}_y$ have been proven in \cite[Section 2]{EFK4}. 
The Baxter operators are compact, hence bounded, for all $y\in\BC\setminus\{z_1,\dots,z_{n-1}\}$.
They are furthermore self-adjoint, and mutually commuting,
\[
\big[\,\mathbf{Q}_{y}\,,\mathbf{Q}_{y'}\,\big]=0,\qquad \forall\; y,y'\in\BC\setminus\{z_1,\dots,z_{n-1}\}.
\] 
The spectral problem 
for the family of Baxter operators is therefore well-posed, and the abstract spectral 
theorem (\cite[Corollary 2.32]{EFK4}) asserts existence of a discrete set $\CS$ such that
\begin{equation}\label{specthm}
\CH_i\simeq \bigoplus_{\chi\in\CS} \CH_{\chi},
\end{equation}
where $\CH_{\chi}$ are the eigenspaces for fixed eigenvalues $\chi$, represented by 
functions $\chi(y,\bar{y})$ on $C$ such that
$
\mathbf{Q}_y\Psi_\chi=\chi(y,\bar{y})\Psi_\chi.
$
The spaces $\CH_\chi$ are finite-dimensional for all $\chi\in\CS$.

\cite[Proposition 3.10]{EFK4} clarifies the relation between Baxter operators and Hamiltonians,
which can be concisely formulated as the operator differential equation
\begin{equation}\label{univoper}
(\pa_y^2+\mathsf{T}(y))\mathbf{Q}_y=0,\qquad
\mathsf{T}(y)=\sum_{r=1}^{n-1}\bigg(\frac{\de_r}{(y-z_r)^2}+\frac{\SH_r}{y-z_r}\bigg),
\end{equation}
where $\de_r=j_r(j_r+1)$.
We refer to \cite{EFK4} for the details. 
This implies that eigenstates of the Baxter operators are simultaneously eigenstates of $\SH_r$, $r=1,\dots,n-1$.

\subsection{Formulation of the analytic Langlands correspondence}

The analytic Langlands correspondence describes the spectrum of Hitchin systems in geometric terms.
It claims that there is a one-to-one correspondence between eigenstates and real opers.

\begin{defn}
An oper on $C_{0,n}=\mathbb{P}^1\setminus\{z_1,\dots,z_n\}$ is a differential 
operator having the form  $\pa_y^2+t(y)$ on $\BC\setminus\{z_1,\dots,z_n\}$, with $t(y)$ of the form
\begin{equation}\label{t-exp-E}
t(y)=\sum_{r=1}^{n}\bigg(\frac{\de_r}{(y-z_r)^2}+\frac{E_r}{y-z_r}\bigg), \qquad E_r\in\BC, \quad r=1,\dots,n,
\end{equation}
having a finite limit $\lim_{y\ra\infty}y^{4}t(y)$. 
 An oper on $C_{0,n}$ is a real oper if it has real monodromy, meaning that it has monodromy group conjugate to $\mathrm{SL}(2,\BR)$.
 
\end{defn}

An important role will be played by the solutions to the system of differential equations
\begin{equation}\label{oper-eqns}
(\pa_y^2+t(y))\chi(y,\bar{y})=0,\qquad
(\pa_{\bar{y}}^2+\bar{t}(\bar{y}))\chi(y,\bar{y})=0.
\end{equation}
An oper admits solutions $\chi$ that are single-valued on $C_{0,n}$ if and only if it is real \cite[Theorem 3.5]{Fi}. 
The solution 
$\chi$ is determined uniquely up to scaling by the oper $\pa_y^2+t(y)$, and vice-versa. 
Here comes a sketch of the 
argument from \cite{Te,Fi}.  
A general solution to \rf{oper-eqns} can locally be represented in the form $\chi(y,\bar{y})=\eta(y)\cdot C\cdot (\eta(y))^{\dagger}$, 
with $\eta(y)$ being the row vector formed out of two linearly independent solutions to the left equation in \rf{oper-eqns}, and $C$ being 
a two-by-two matrix with complex matrix elements. 
By a change of basis one may bring $C$ to diagonal form $C=\mathrm{diag}(1,C_{22})$, with $C_{22}\in\{-1,0,1\}$. Single-valued 
solutions having $C_{22}=1$ or $C_{22}=0$ would define metrics with positive or vanishing constant curvature on $C$, leading 
to contradictions with the Gauss-Bonnet theorem. Single-valuedness in the remaining case $C_{22}=-1$ requires 
that the monodromy is in $\mathrm{SU}(1,1)$, and therefore conjugate to $\mathrm{SL}(2,\BR)$.

The analytic Langlands correspondence is the content of the following theorem. 
\begin{thm}\label{thm-RealLang}
The eigenspaces $\CH_{\chi}$ are one-dimensional, and the 
set $\CS$ is in bijection to the set $\mathrm{Op}_{\BR}^{}(C)$ of all real opers on $C$.
\end{thm}

\section{Construction of the SoV transformation}



The main tool for the proof of Theorem \ref{thm-RealLang} will be an integral
transformation referred to as Separation of Variables ($\mathsf{SoV}$) transformation that will be defined 
in this section. It serves to simplify the eigenvalue problems for Hitchin's Hamiltonians.

Let us define $\CH_{\rm Skl}^{}=L^2(\BC^{n-3},d\mu_{\rm Skl}^{})$, with 
\begin{equation}\label{def-muSkl}
d\mu_{\rm Skl}^{}(\bm{y})=d^{2(n-3)}\bm{y}\,
\prod_{\substack{r,s=1\\ r<s}}^{n-1}|z_r-z_s|^{-2}\prod_{\substack{k,l=1\\ k<l}}^{n-3}
|y_k-y_l|^2\prod_{r=1}^{n-1}\frac{1}{|\kappa_r(\bm{y})|^2},\end{equation}
where 
\begin{equation}
\quad \kappa_r(\bm{y}):=
\frac{\prod_{k=1}^{n-3}(z_r-y_k)}{\prod_{s\neq r}(z_r-z_s)},\qquad r=1,\dots,n-1.
\end{equation}
We will construct the  $\mathsf{SoV}$ transformations $\SS_{\bm{j},i}:\CH_i\ra\CH_{\rm Skl}^{}$,
$\bm{j}=(j_1,\dots,j_{n})$, $i=0,1$,
as the composition
\[
\SS_{\bm{j},i}=\mathsf{SoV}\circ \mathsf{F}_{\bm{j},i},
\]
where $\mathsf{F}_{\bm{j},i}$ is a twisted version 
of the Fourier transformation of homogeneous translation-invariant 
functions defined below, and  $\mathsf{SoV}$ is defined by a change of variables defined below.

\subsection{Fourier transformation}

We note that homogenous locally integrable functions $\Psi$ of $d+1$ variables $x_2,\dots,x_{d+2}$ naturally define 
homogenous regular distributions $D_{\Psi}$ on the Schwartz space 
$\CS(\BC^{d+1})$  by means of
\begin{equation}\label{regdist}
\langle D_{\Psi},f\rangle=\int_{\BC^{d+1}}d^{2(d+1)}\bm{x}\; \Psi(\bm{x})f(\bm{x}), \qquad 
f\in \CS(\BC^{d+1}).
\end{equation}
The Fourier transformation  of $D_{\Psi}$ can be defined in the distributional sense,
yielding a homogenous distribution 
$\widetilde{D}_\Psi$.
It turns out that the homogenous distribution 
$\widetilde{D}_\Psi$ is regular, meaning that there exists a 
a homogenous function $\widetilde{\Psi}$ such that $\widetilde{D}_{\Psi}=D_{\widetilde{\Psi}}$.
Let $\eta_\rho:\BR_+\ra \BR$ be a family of smooth functions satisfying $\eta_\rho(x)=1$ for $x<\rho$ and $\eta_\rho(x)=0$
for $x>2\rho$.
\begin{lem}\label{FT-lem}
$\widetilde{\Psi}$ can be represented for smooth homogenous functions $\Psi$ as the limit
\begin{equation}\label{FT-def}
\widetilde{\Psi}({\bm{k}})=
\lim_{\rho\ra\infty}\int_{\BC^{d+1}}d^{2(d+1)}{\bm{x}}\; 
{\Psi}(\bm{x})\,\eta_\rho^{}(\lVert {\bm{x}}\rVert)\,\prod_{r=2}^{d+2}
{e^{-2\mathrm{i}\,\mathrm{Im}(k_rx_r)}}, 
\end{equation}
where  ${\bm{k}}=(k_2,\dots,k_{d+2})$. In this case, \rf{FT-def}
defines a smooth function $\widetilde{\Psi}({\bm{k}})$.
\end{lem}

The proof of \cite[Lemma 1]{GS} can easily be adapted to our case where $d=n-3$.
We  note that 
\begin{equation}\label{scalingtildePsi}
\widetilde{\Psi}(\la{\bm{k}})=|\la|^{2\bar{J}} \, \widetilde{\Psi}({\bm{k}}),\qquad \bar{J}:=-J-d-1=-\frac{d+1}{2}-iS.
\end{equation}
The function 
$
\widetilde{\Psi}({\bm{k}})
$ 
therefore defines a homogenous regular distribution on $\BC^{d+1}$, represented by functions 
$\widetilde{\Psi}(k_2,\dots,k_{n-1})$ of the form
\begin{equation}\label{Psi-vs-psi}
\widetilde{\Psi}(k_2,\dots,k_{n-1})=|k_{n-1}|^{2\bar{J}}\, \widetilde{\psi}(k_2/k_{n-1},\dots,k_{n-2}/k_{n-1}).
\end{equation}
It follows from Lemma \rf{FT-lem} that the Fourier transformation maps the space $\CS_J$ of 
smooth sections of $\CL_J$ to the space $\CS_{-J-d-1}$. 
The definition of the norm given in \rf{Psi-norm} can be applied to $\widetilde{\Psi}$. 

\begin{thm}\label{F-unitarity}
The Fourier transformation sending $\psi\in \CS_{J}$ to $\widetilde{\psi}\in\CS_{-J-d-1}$ is unitary.
\end{thm}

The proof in \cite[Section 3]{GS} can be adapted to 
our case.
This implies that the maps $\mathsf{F}_{\bm{j},i}$ have natural extensions from $\CS_{J_i}$ to the spaces 
$\CH_{J_i}$ of square-integrable sections of $\CL_{J_i}$.

\subsection{Separation of Variables}

A key role will be played by the change of variables $k_r=k_r(\bm{y},\rho)$ defined by 
\begin{equation}\label{ch-of-var-SoV}
\sum_{r=1}^{n-1}\frac{
k_r(\bm{y},\rho)}{y-z_r}=\rho\,\frac{\prod_{k=1}^{n-3}(y-y_k)}{\prod_{r=1}^{n-1}{(y-z_r)}}.
\end{equation}
In order to describe the induced maps between 
spaces of functions, we shall use ``twisted'' versions $\mathsf{F}_{\bm{j},i}$
of the Fourier transformation, sending $\Psi$ to the functions $\Phi$, defined as
\begin{equation}
\Phi(\bm{k})=\bigg|\sum_{r=1}^{n-1}z_rk_r\bigg|^{-2(j_{n,i}+1)}\prod_{r=1}^{n-1}|k_r|^{2j_r+2}\, \widetilde{\Psi}(\bm{k}),
\end{equation}
for $i=0,1$, using the notations $j_{n,i}$ introduced in \rf{jidef}.  
We may then note the simple identity
\begin{equation}\label{Phi-Phi}
\Phi(\bm{k}(\bm{y},\rho))=\Phi(\bm{\kappa}(\bm{y})), \qquad \bm{\kappa}(\bm{y}):=\frac{1}{\rho}\bm{k}(\bm{y},\rho),
\end{equation}
following from   \rf{scalingtildePsi} and $\sum_{r=1}^{n-1}z_rk_r(\bm{y},\rho)=\rho$.
The identity \rf{Phi-Phi} allows  us to define 
\begin{equation}\label{tildePsi-to-Phi}
\mathsf{SoV}({\Phi})(\bm{y}):=\phi(\bm{y}),\qquad \phi(\bm{y}):=\Phi(\bm{\kappa}(\bm{y})).
\end{equation}
We then have

\begin{lem}\label{SOV-unitarity}
The map $\mathsf{SoV}:{\CH}_{-J-d-1}\ra \CH_{\rm Skl}^{}$ 
defined above is unitary.
\end{lem}
\begin{proof}We note that 
the change of variables defined by 
\begin{equation}
\la_r(\bm{y}):=\frac{\kappa_r(\bm{y})}{\kappa_{n-1}(\bm{y})},\quad r=2,\dots,n-2, 
\end{equation}
induces the following change of measures
\begin{equation}\label{change-to-Sklyanin}
d^{2(n-3)}\bm{\lambda} =
\frac{d^{2(n-3)}\bm{y}}{|\kappa_{n-1}(\bm{y})|^{2(n-2)}} 
\prod_{\substack{r,s=1\\ r<s}}^{n-1}|z_r-z_s|^{-2}
\prod_{\substack{k,l=1\\ k<l}}^{n-3}
|y_k-y_l|^2.
\end{equation}
It furthermore follows from \rf{Psi-vs-psi} and \rf{tildePsi-to-Phi} that
\begin{equation}\label{relabsval}
\big|\psi\big(\bm{\lambda}(\bm{y})\big)\big|^2=|\phi(\bm{y})|^2\,
{|\kappa_{n-1}(\bm{y})|^{2(n-2)}} \prod_{r=1}^{n-1}\frac{1}{|\kappa_r(\bm{y})|^2}.
\end{equation}
By combining \rf{change-to-Sklyanin} and \rf{relabsval} it follows that
\begin{equation}
\int_{\BC^{n-3}}d^{2(n-3)}\bm{\lambda} \;|\psi(\bm{\lambda})|^2=
\int_{\BC^{n-3}}d\mu_{\rm Skl}^{}(\bm{y})\;|\phi(\bm{y})|^2,
\end{equation}
which is what we wanted to prove.
\end{proof}
Unitarity of the map $\mathsf{SoV}$ follows by
combining Theorem \ref{F-unitarity} and Lemma \ref{SOV-unitarity}. 

For any given differential operator $\mathsf{D}$ of the form
\[
\mathsf{D}(z)=\sum_{\al\in\CI}f_{\al}(z)\mathsf{D}^{(\al)},
\]
with $\{\mathsf{D}^{(\al)};\al\in\CI\}$ being a finite set of  differential operators
in the variables $k_r$, and $f_{\al}$, $\al\in\CI$ being holomorphic functions on $C=\mathbb{P}^1\setminus\{z_1,\dots,z_n\}$,
we may define a  differential operator $\mathsf{D}(y_k)$ by
substituting $z$ by $y_k$ ``from the left'', and pullback under \rf{ch-of-var-SoV}, more precisely 
\begin{equation}
{\mathsf{D}(y_k)}:=
\sum_{\al\in\CI}f_{\al}(y_k)\mathsf{D}^{(\al)}_{\bm{y},\rho},
\end{equation}
with $\mathsf{D}^{(\al)}_{\bm{y},\rho}$, $\al\in\CI$, being the differential operators in the variables $y_1,\dots,y_d$ and $\rho$
associated to $\mathsf{D}^{(\al)}$
by the change of variables \rf{ch-of-var-SoV}.

A crucial property of the map $\mathsf{SoV}$
can now be formulated as follows.
\begin{propn}\label{SoV-propn}
The Fourier transformation $\widetilde{\mathsf{T}}(z):=\mathsf{F}_{\bm{j},i}\cdot\mathsf{T}(z)\cdot\mathsf{F}^{-1}_{\bm{j},i}$ 
of $\mathsf{T}(z)$ satisfies
\begin{equation}\label{SoV-intertw}
\widetilde{\mathsf{T}}(y_k)=-\pa^2_{y_k}.
\end{equation}
\end{propn}
By composing \rf{SoV-intertw} with the Fourier transformations $\mathsf{F}_{\bm{j},i}$, it follows from
\rf{SoV-intertw} that $\mathsf{S}_{\bm{j},i}$ intertwines the eigenvalue equations for the
Hamiltonians $\SH_r$, or equivalently the equations $\mathsf{T}(z)\Psi=t(z)\Psi$,
with the decoupled  system of equations 
\begin{equation}
-\pa_{y_k}^2\Phi(\bm{y})=t(y_k)\Phi(\bm{y}), \qquad k=1,\dots,d.
\end{equation}

This fact has been discovered in a closely related context in \cite{Skl}. It motivates the terminology ``Separation of Variables''.
A self-contained proof 
is given below. 

\begin{proof}
As a preparation, 
let us note that $\widetilde{\mathsf{T}}(z)$ can 
be more explicitly represented as 
\begin{equation}\label{tildeTdef}
\widetilde{\mathsf{T}}(z)=-(\mathsf{a}(z))^2+\mathsf{a}'(z)-\mathsf{c}(z)\mathsf{b}(z),
\end{equation}
where 
\begin{equation}\label{abcdef}
\mathsf{c}(z)=\sum_{r=1}^{n-1}\frac{k_r}{z-z_r} ,
\quad
\mathsf{a}(z)=\sum_{r=1}^{n-1}\frac{-k_r}{z-z_r} \frac{\pa}{\pa k_r},
\quad
\mathsf{b}(z)=\sum_{r=1}^{n-1}\frac{1}{z-z_r} \bigg(\frac{\de_r}{k_r}-k_r\frac{\pa^2}{\pa k_r^2}\bigg).
\end{equation}
We may observe that equations \rf{ch-of-var-SoV} imply
\[
\frac{\pa}{\pa y_k}F(\bm{k}(\bm{y},\rho))=\sum_{r=1}^{n-1}\frac{\pa k_r}{\pa y_k}\frac{\pa }{\pa k_r}F(\bm{k}(\bm{y},\rho))
=\sum_{r=1}^{n-1}\frac{k_r}{y_k-z_r}\frac{\pa }{\pa k_r}F(\bm{k}(\bm{y},\rho)),
\]
yielding the key identities
\begin{equation}\label{key-SOV}
{{\mathsf{a}}(y_k)}
=-\frac{\pa}{\pa y_k},\qquad k=1,\dots,n-3.
\end{equation}
Equations \rf{key-SOV} furthermore imply 
\begin{equation}\label{secondkey-SOV}
{{\mathsf{a}}^2(y_k)}=\frac{\pa}{\pa y_k}{\mathsf{a}}(y_k)
=\frac{\pa^2}{\pa y_k^2}+{\mathsf{a}}'(y_k).
\end{equation}
With the help of  \rf{tildeTdef} and \rf{secondkey-SOV} it is now easy to see that
\begin{equation}
\widetilde{\mathsf{T}}(y_k)=-\frac{\pa^2}{\pa y_k^2},\qquad k=1,\dots,n-3.
\end{equation}
This is equivalent to \rf{SoV-intertw}.\end{proof}



\subsection{The inverse transformation}

Unitarity of the SoV-transformation ensures existence of an inverse transformation. It can be realised explicitly 
by inverting the relation \rf{tildePsi-to-Phi}, taking advantage of the fact that equations \rf{ch-of-var-SoV} can be used to
define functions $y_k=y_k(\bm{k})$ up to permutations $y_k\mapsto y_l$, $y_l\mapsto y_k$.
The resulting map $\mathsf{SoV}^{-1}$ takes elements $\phi$ 
of $\CH_{\rm Skl}^{}$ to scaling-invariant functions $\Phi(\bm{k})$ of $n-2$ variables $\bm{k}=(k_2,\dots,k_{n-1})$.

It next follows from \rf{SoV-intertw} that $\mathsf{S}_{\bm{j},i}^{-1}$ maps 
elements $\phi_{\chi}$ of  $\CH_{\rm Skl}^{}$ having the form 
\begin{equation}
\phi_{\chi}(\bm{y})=\prod_{k=1}^{n-3}\chi(y_k,\bar{y}_k),
\qquad
\begin{aligned} &
(\pa_y^2+t(y))\chi(y,\bar{y})=0,\\
&(\pa_{\bar{y}}^2+\bar{t}(\bar{y}))\chi(y,\bar{y})=0,
\end{aligned}
\end{equation}
to eigenfunctions $\Psi_{\chi,i}$ of the differential operators
$\mathsf{H}_r$, with eigenvalues $E_r$ determined by $t(y)$ using \rf{t-exp-E}. 
We shall next describe the main analytic properties of the functions $\Psi_{\chi,i}=\mathsf{S}_{\bm{j},i}^{-1}(\phi_{\chi})$.

\subsubsection{Square-integrability}

Square-integrability of the functions $\Psi_{\chi,i}=\mathsf{S}_{\bm{j},i}^{-1}(\phi_{\chi})$ will follow with the help of the 
unitarity of the SoV-transformation from the following lemma. 
\begin{lem} 
The functions $\phi_{\chi}$ satisfy
\[
\int_{\BC^{n-3}}d\mu_{\rm Skl}^{}(\bm{y})\;|\phi_{\chi}(\bm{y})|^2<\infty,
\] 
and therefore represent
elements of $\CH_{\rm Skl}^{}$.
\end{lem}
\begin{proof} This can be verified using \rf{def-muSkl}, and taking into account that 
$\chi$ is real analytic on $C$, having singular behaviour of the form 
\begin{equation}\label{sing-chi}
\chi(y,\bar{y})= \sum_{\ep=\pm 1}|y-z_r|^{1+\ep(2j_r+1)}C_{\ep}^{(r)}(\bm{j})\big(1+\CO(|y-z_r|)\big),
\end{equation}
near the punctures $z_r$, $r=1,\dots,n$. By using the explicit form of $d\mu_{\rm Skl}^{}(\bm{y})$
given in \rf{def-muSkl}
one may easily verify that $d\mu_{\rm Skl}^{}(\bm{y})\;|\phi_{\chi}(\bm{y})|^2$ is integrable 
with respect to $y_k$ near $y_k=z_r$ for all $k=1,\dots,n-3$, and $r=1,\dots,n-1$, and furthermore
\begin{align}\label{Sklmeas-infty}
&d\mu_{\rm Skl}^{}(\bm{y})\,|\phi_{\chi}(\bm{y})|^2=\\[-1ex]
&\qquad=|y_k|^{-6} \cdot |y_k|^{4}\,d^2\bigg(\frac{1}{y_k}\bigg)\cdot
\Bigg|\sum_{\ep=\pm 1}C_{\ep}^{(n)}(\bm{j})\,|y_k|^{1-\ep(2j_n+1)}\big(1+\CO(y_k^{-1})\big)\Bigg|^2
d\check{\mu}_{\rm Skl}^{}(\check{\bm{y}}),\notag\end{align}
with $\check{\bm{y}}=(y_1,\dots,y_{k-1},y_{k+1},\dots,y_{n-3})$, and $d\check{\mu}_{\rm Skl}^{}(\check{\bm{y}})$
obtained from $d\mu_{\rm Skl}^{}(\bm{y})$ by omitting all factors containing $y_k$. Integrability with respect to $y_k$ near 
$z_n=\infty$ now follows from \rf{Sklmeas-infty}.
\end{proof}

\begin{rem} A coordinate-free description of $\mu_{\rm Skl}^{}(\bm{y})$ has been proposed in \cite[Section 3.3.2]{GT}.
\end{rem}

\subsubsection{Real-analyticity}

One should note, however, that the $L^2$-functions $\Psi_{\chi,i}=\mathsf{S}_{\bm{j},i}^{-1}(\phi_{\chi})$ are not necessarily 
point-wise defined, being defined by the Fourier transformation of homogenous distributions.

\begin{lem}
The distributions $\Psi_{\chi,i}$ solve the holonomic system of differential equations
\[
\mathsf{H}_r \Psi_{\chi,i}=E_r\Psi_{\chi,i}, \quad r=2,\dots,n-2, \quad 
\sum_{r=1}^{n-1}(x_r\pa_{x_r}-j_r)\Psi_{\chi,i}=0,\quad
\sum_{r=1}^{n-1}\pa_{x_r}\Psi_{\chi,i}=0,
\]
together with the complex conjugate equations. 
\end{lem}
Elliptic regularity for holonomic systems implies that $\Psi_{\chi,i}$ can be represented by
functions $\Psi_{\chi,i}(\bm{x})$ which are real-analytic away from the singular locus of the differential operators 
$\mathsf{H}_r$. 

\begin{rem} 
The points in the singular locus  correspond to the so-called wobbly bundles 
\cite{EFK}. 
However, one should note that the type of singular behaviour that can occur at the wobbly locus is strongly  constrained by the unitarity of the SoV transformation. As $\Psi_{\chi,i}$ are in the image of 
$\CH_{\rm Skl}^{}$ under the unitary operators $\mathsf{S}_{\bm{j},i}^{-1}$, it is ensured that the singularities of $\Psi_{\chi,i}$ 
can not spoil
square-integrability. The unitarity of the SoV-transformation allows us to avoid a detailed 
analysis of the singularities that may occur at the wobbly loci.  \end{rem}

\subsubsection{Meromorphic continuation}

Let us finally note another property of the inverse transformation $\mathsf{S}_{\bm{j},i}^{-1}(\phi_{\chi})$
that will become important  below. 
Considering the dependence of the  wave-functions $\Psi_{\chi,i}$ with respect to the parameters
$j_1,\dots,j_n$, one may note that this dependence is meromorphic. In order to see this, one may 
first note that the dependence of $\chi$ on $j_1,\dots,j_n$ is analytic, as $\chi$ is the solution to a 
system of differential equations having coefficients depending analytically on $j_1,\dots,j_n$.
The regular distribution 
associated to the function $\widetilde{\Psi}_{\bm{E}}(\bm{k})$
defined by inverting the relation \rf{tildePsi-to-Phi} also depends 
analytically on $j_1,\dots,j_n$. The inverse Fourier transformation $\mathsf{F}_{\bm{j}, i}^{-1}$ is meromorphic in 
$J$, as can be seen by an argument very similar to the one in \cite[Section 2]{GS}.

We may therefore consider the meromorphic continuation of wave-functions $\Psi_{\chi,i}$ away from the 
range $j_r\in -\frac{1}{2}+\mathrm{i}\BR$ in which the unitarity of the SoV-transformation holds.

\section{Outline of the proof of the analytic Langlands correspondence}


Theorem \ref{thm-RealLang} will follow from the following two results. 
\begin{itemize}
\item[i)] ({\it Separation of variables})
Any eigenstate $\Psi_\chi\in\CH_i$ 
of $\SH_w$ defines an element $\phi_\chi=\mathsf{S}_{\bm{j},i}\Psi_\chi$ of $\CH_{\rm SoV}^{}$. 
The wave-function $\phi_\chi(\bm{y})$ 
representing $\phi_\chi$ has the form 
\begin{equation}\label{Phi-chi}
\phi_\chi(\bm{y})=\prod_{k=1}^{n-3}\chi(y_k,\bar{y}_k),
\end{equation}
where $\chi(y,\bar{y})$ satisfies
\[
(\pa_{y}^2+t_\chi(y))\chi(y,\bar{y})=0, \qquad (\bar{\pa}_{\bar{y}}^2+\bar{t}_\chi(\bar{y}))
\chi({y},\bar{y})=0.
\]
Single-valuedness of $\chi$ implies that 
$\pa_{y}^2+t_\chi(y)$ must be a real oper. 
\item[ii)] ({\it Hecke eigenvalue property})
To each real oper $\pa_w^2+t(w)$ we may associate the single-valued solution $\chi$, unique up to scaling, of the pair of differential equations
\begin{equation}
(\pa_w^2+t(w))\chi(y,\bar{y})=0, \qquad (\bar{\pa}_{\bar{w}}^2+\bar{t}(\bar{w}))\chi(y,\bar{y})=0. 
\end{equation}
The functions $\mathsf{S}_{\bm{j},i}^{-1}\phi_{\chi}$ associated to $\chi$,
with $\phi_\chi$ defined in \rf{Phi-chi},
are
eigenstates of $\SQ_{w}$ with eigenvalue $\chi(w,\bar{w})$.
\end{itemize}

In order to prove these two results, we shall consider the 
SoV transformations $\mathsf{S}_{\bm{j}',i}^{(n+1)}$, associated to $C_{0,n+1}=\mathbb{P}^1\setminus\{z_0,\dots,z_{n}\}$,
$\bm{j'}=(j_0,j_1,\dots,j_n)$. 
The definition of $\mathsf{S}_{\bm{j}',i}^{(n+1)}$ is completely 
analogous to the definition of $\mathsf{S}_{\bm{j},i}^{(n)}:=\mathsf{S}_{\bm{j},i}^{}$ given in the previous section, 
assigning index $0$ to all variables associated to the additional puncture $z_0$, and admitting the factorisation 
$\mathsf{S}_{\bm{j}',i}^{(n+1)}=\mathsf{SoV}^{(n+1)}\circ \mathsf{F}_{\bm{j},i}^{(n+1)}$.
In the special case $j_0=-1$, we will exhibit relations between $\mathsf{S}_{\bm{j}',i}^{(n+1)}$ and 
the composition of $\mathsf{S}_{\bm{j},i}^{(n)}$ with the Hecke operators, implying the 
results i), ii) formulated above.

\subsection{SoV-transform of eigenstates}

The interplay between SoV-transform and Hecke operators will prove one direction of the analytic Langlands
correspondence rather easily.

\begin{propn}
The SoV-transform $\phi_\chi=\mathsf{S}_{\bm{j},i}^{(n)}\Psi_\chi$ of an eigenstate $\Psi_\chi$ can be represented in terms of the eigenvalue $\chi(z_0,\bar{z}_0)$ of the Baxter operators $\mathsf{H}_{z_0}$ as
\[
\phi_\chi(\bm{y})=\prod_{k=1}^{n-3}\chi(y_k,\bar{y}_k).
\]
\end{propn}
\begin{proof}
Let us note that 
\begin{equation}\label{Hecke-Fourier}
(\mathsf{F}^{(n)}_{\bm{j},i}\SH_{z_0}\Psi)(\bm{k})=
\lim_{k_0\ra 0}(\mathsf{F}^{(n+1)}_{\bm{j}',i}\Psi')(\bm{k}'),\qquad 
\Psi'(\bm{x}'):=\mathbb{H}_{z_0,x_0}\Psi(\bm{x}),
\end{equation}
using the notations $\bm{x}'=(x_0,\dots,x_{n-1})$ and $\bm{k}'=(k_0,\dots,k_{n-1})$.
When $\Psi\equiv\Psi_\chi$ is an eigenstate of $\SH_{z_0}$ with eigenvalue 
$\chi(z_0,\bar{z}_0)$, we may conclude from \rf{Hecke-Fourier} that
\begin{equation}\label{Hecke-eigen-k}
\chi(z_0,\bar{z}_0)(\mathsf{F}^{(n)}_{\bm{j},i}\Psi_\chi)(\bm{k})
=\lim_{k_0\ra 0}(\mathsf{FT}^{(n+1)}_{\bm{j}',i}\Psi'_\chi)(\bm{k}').
\end{equation}
The change of variables $\mathsf{SoV}^{(n+1)}$,  explicitly given by
\begin{equation}\label{SoV-(n+1)}
k_0'(\bm{y}')=y\,(z_0-y_0)\frac{\prod_{k=1}^{n-3}(z_0-y_k)}{\prod_{r=1}^{n-1}(z_0-z_s)},
\quad
k_r'(\bm{y}')=y\,\frac{z_r-y_0}{z_r-z_0}k_r(\bm{y}),\quad r=1,\dots,n.
\end{equation}
can be used to define the SoV transform $
\phi'=\mathsf{S}_{\bm{j}',i}^{(n+1)}\Psi'_\chi
$
of the function $\Psi'_\chi$ as before.
The function  
$
\phi'(y_0,\dots,y_{n-3})$
is manifestly symmetric in $y_0,\dots,y_{n-3}$. 

Noting that 
$y_0\ra z_0$ implies that $k_0'(\bm{y}')\ra 0$ and $k_r'(\bm{y}')\ra k_r(\bm{y})$, it follows from
(\ref{Hecke-eigen-k}) that
\begin{equation}\label{limy0}
\lim_{y_0\ra z_0}
\big(\mathsf{S}^{(n+1)}_{\bm{j}',i}\Psi'_\chi\big)(\bm{y}')=\chi(z_0,\bar{z}_0)
\big(\mathsf{S}_{\bm{j},i}^{(n)}\Psi_\chi\big)(\bm{y}).
\end{equation}
As the function on the left side of \rf{limy0} is 
symmetric under permutations of the arguments exchanging $z_0$
with $y_k$ for any $k=1,\dots,n-3$, it follows that
\[
\lim_{y_0\ra z_0}\big(\mathsf{S}^{(n+1)}_{\bm{j}',i}\Psi'_\chi\big)(\bm{y}')=\chi(z_0,
\bar{z}_0)\prod_{k=1}^{n-3}
\chi(y_k,\bar{y}_k).
\]
The claim follows immediately. 
\end{proof}
As the SoV-transform can be inverted, we see that eigenstates are uniquely determined by their eigenvalues, 
which means that the spectrum is simple.

\subsection{Hecke eigenvalue property}

We have seen that eigenstates define real opers as their Hecke eigenvalues. 
It remains to prove that all real opers define eigenstates. To this aim, we shall consider inverse SoV transformations of the form $\big(\mathsf{S}_{\bm{j}',i}^{(n+1)}\big)^{-1}\phi'_\chi$,
where
\begin{equation}
\phi'_\chi(\bm{y}')=\chi(y_0,\bar{y}_0)
\phi_\chi(\bm{y}),\quad
\phi_\chi(\bm{y})=\prod_{k=1}^{n-3}\chi(y_k,\bar{y}_k),\quad 
\begin{aligned} &\bm{y}=(y_1,\dots,y_{n-3}),\\
& \bm{y}'=(y_0,y_1\dots,y_{n-3}),
\end{aligned}
\end{equation}
where $\chi(y,\bar{y})$ is the solution to the differential equation associated to a real oper on
$C_{0,n}$ rather than $C_{0,n+1}$.  Opers on
$C_{0,n+1}$ are differential operators of the form $\pa_y^2+t^{(n+1)}(y)$, with 
\begin{equation}
t^{(n+1)}(y)=\sum_{r=0}^{n}\bigg(\frac{j_r(j_r+1)}{(y-z_r)^2}+\frac{E_r}{y-z_r}\bigg).
\end{equation}
We see that opers on $C_{0,n+1}$ represent opers on $C_{0,n}$ if $j_0=-1$ and $E_0=0$.

The following theorem  expresses a crucial link between the inverse Separation of Variables transformation and 
the Hecke modifications. 
\begin{thm}\label{Heckemod-SOV}
The Hecke modification of the states $\big(\mathsf{S}_{\bm{j},1}^{(n)}\big)^{-1}\phi_\chi\in\CH_1$ can be represented as 
\begin{equation}\label{HeckevsSoV(n+1)}
\mathbb{H}_{z_0,x_0}\big(\mathsf{S}_{\bm{j},0}^{(n)}\big)^{-1}\phi_\chi
=\frac{1}{\chi_n^{}}\big(\mathsf{S}_{\bm{j}',1}^{(n+1)}\big)^{-1}\phi'_\chi,\qquad 
\chi_n^{}:=\frac{\Ga(-2j_n)}{\Ga(2j_n+1)}\,C_{-}^{(n)}(\bm{j}), 
\end{equation}
referring to \rf{sing-chi} for the definition of $C_{-}^{(n)}(\bm{j})$.
\end{thm}
This theorem is proven in Section \ref{Heckemod-SOV-pf}. It will be the key to one of our main results:

\begin{thm}\label{Heckeop-SOV}
The states $\psi_{\chi}=\big(\mathsf{S}_{\bm{j},i}^{(n)}\big)^{-1}\phi_\chi\in\CH_1$ are eigenstates of the operators $\mathbf{Q}_{w}$,
\begin{equation}\label{Hecke-eigen}
\mathbf{Q}_{w}\,\psi_\chi =
\frac{1}{\chi_n^{}}\chi(w,\bar{w})\,\psi_\chi.
\end{equation}
\end{thm}
The proof will be given in Section \ref{Heckeop-SOV-pf}.

\section{Hecke modifications of the inverse SoV transformation}\label{Heckemod-SOV-pf}

We shall here prove the crucial Theorem \ref{Heckemod-SOV}. 
We will consider the inverse $\mathsf{SoV}$-transformations associated to Riemann spheres $C_{0,n+1}$ obtained from 
$C_{0,n}$ by adding another marked point at $z=z_0$. When the parameter $j_0$ associated to $z_0$ is chosen
as $j_0=-1$, we will find a relation with the Hecke modification  $\mathbb{H}_{z_0,x_0}$.

We begin by introducing the relevant notations. 
A main ingredient of the integrand
will be a single-valued function 
$\chi_{}(y,\bar{y})$ satisfying 
\begin{equation}\label{chiopdiff}
\begin{aligned}
&(\pa_y^2+t(y))\chi_{}(y,\bar{y})=0,\\ 
&(\pa_{\bar{y}}^2+\bar{t}(\bar{y}))\chi_{}(y,\bar{y})=0,
\end{aligned}\qquad
t(y)=\sum_{r=1}^{n-1}\bigg(\frac{\de_r}{(y-z_r)^2}+\frac{E_r}{y-z_r}\bigg).
\end{equation}
in both cases. We shall consider the following integrals
\begin{align}\label{Fourier(n+1)}
\Sigma^{(n+1,1)}_{\chi}\Big(\begin{smallmatrix} j_0 & j_1 & \dots & j_{n-1}\\[.75ex]
x_0 & x_1 & \dots &x_{n-1}\end{smallmatrix};j_{n}\Big)=\int\prod_{r=0}^{n-1}
\frac{d^2k_r\,e^{2\mathrm{i}\,\mathrm{Im}(k_rx_r)}}{\pi\,|k_r|^{2j_r+2}}&
\pi\de^{(2)}\big({\textstyle \sum_{r=0}^{n-1}k_r}\big)\\[-1.5ex] 
&\times|\hat{\rho}(\bm{k})|^{-2j_{n}}
\prod_{k=0}^{n-3}\chi_{}^{}\big(y_k(\bm{k}),\bar{y}_k(\bm{k})\big), 
\notag\end{align}
with $y_0(\bm{k}),\dots,y_{n-3}(\bm{k})$ and $\hat{\rho}(\bm{k})$, $\bm{k}=(k_0,\dots,k_{n-1})$, defined by 
\begin{equation}\label{y-to-k}
k_r=\hat{\rho}_i(\bm{k})
\frac{\prod_{k=0}^{n-3}(z_r-y_k(\bm{k}))}{\prod_{s\neq r}(z_r-z_s)}, \qquad r=0,\dots,n-1,\qquad\hat{\rho}(\bm{k})=
\sum_{r=0}^{n-1}z_rk_r.
\end{equation}

Theorem \ref{Heckemod-SOV} is equivalent to the following integral  identity:
\begin{equation}\label{Heckemod-SOV-repn}
\begin{aligned}
&\Sigma^{(n+1,1)}_{\chi}\Big(\begin{smallmatrix} -1 & j_1 & \dots & j_{n-1}\\[.75ex]
x_0 & x_1 & \dots &x_{n-1}\end{smallmatrix};j_n\Big)=
\,\chi_n^{}\,\prod_{r=1}^{n-1}\frac{|x_r-x_0|^{4j_r}}{|z_r-z_0|^{2j_r}}\,
\Sigma^{(n,0)}_{{\chi}}\Big(\begin{smallmatrix} j_1 & \dots & j_{n-1}\\[.25ex]
-\frac{z_1-z_0}{x_1-x_0} & \dots &-\frac{z_{n-1}-z_0}{x_{n-1}-x_0}\end{smallmatrix};j_n\Big),
\end{aligned}
\end{equation}
with function $\Sigma^{(n,0)}_{{\chi}}=\big(\mathsf{S}_{\bm{j},0}^{(n)}\big)^{-1}\phi_\chi$ on the right side 
of \rf{Heckemod-SOV-repn} being represented by an integral obtained from \rf{Fourier(n+1)}
by restricting the ranges for $r$ and $k$ to $\{1,\dots,n-1\}$ and $\{1,\dots,n-3\}$, respectively, and
replacing $j_n$ by $-j_n-1$.
In order to prove \rf{Heckemod-SOV-repn} we will first derive a differential equations for the function 
$\Sigma^{(n+1,1)}$ on the left side of \rf{Heckemod-SOV-repn}, then note that the right side of \rf{Heckemod-SOV-repn}
satisfies the same equations, and finally establish the equality of the limits $z_0\ra 0$, 
followed by $x_0\ra 0$.

\subsection{Differential equation}

It will be crucial to observe that one of the eigenvalue equations satisfied by $\Theta^{(n+1,1)}_\chi$ simplifies
to a first order differential equation when $j_0=-1$, and $E_0=0$.
 \begin{propn}\label{Nulleqpropn}
 When $j_0=-1$ and $E_0=0$, the distribution $\Theta^{(n+1,1)}_\chi$ satisfies 
 \begin{equation}\label{nulleqn}
\mathcal{N}\Theta^{(n+1,1)}_\chi=0, \qquad \mathcal{N}=\sum_{r=1}^{n}\frac{1}{z_0-z_r}\bigg((x_r-x_0)^2\frac{\partial}{\partial x_r}-2j_r(x_r-x_0)\bigg).
\end{equation}
\end{propn}
\begin{proof}
The condition $E_0=0$ implies that
\[
0=\,\Big(\underset{z=z_0}\Res \widetilde{\mathsf{T}}(z)\Big)\,\widetilde{\Theta}^{(n+1,1)}_\chi=-k_0\,\widetilde{\mathcal{N}}\,\widetilde{\Theta}^{(n+1,1)}_\chi,
\]
with $\widetilde{\Theta}^{(n+1,1)}_\chi$ being the Fourier transformation of ${\Theta}^{(n+1,1)}_\chi$, 
with   $\widetilde{\mathsf{T}}(z)$ defined in \rf{tildeTdef}, and 
\[
\widetilde{\mathcal{N}}:=\sum_{r=1}^{n-1}\frac{k_r}{z_0-z_r}\bigg(
\frac{\pa^2}{\pa k_r^2}-\frac{\de_r}{k_r^2}-2\frac{\pa^2}{\pa k_r\pa k_0}+\frac{\pa^2}{\pa k_0^2}\bigg).
\]
It follows that $\widetilde{\mathcal{N}}\widetilde{\Theta}^{(n+1,1)}_\chi=0$, implying \rf{nulleqn} by inverse Fourier transformation. 
\end{proof}

\begin{rem} Equation \rf{nulleqn} can be recognised as the differential equation characterising 
Hecke modifications of the conformal blocks of affine Lie algebras \cite{AFW}.
\end{rem}

It follows from the differential equations that there must exist a function $\Psi_\chi$ satisfying 
\rf{trans-dil-Psi} such that
\begin{equation}
\begin{aligned}
&\Sigma^{(n+1,1)}_{\chi}\Big(\begin{smallmatrix} -1 & j_1 & \dots & j_{n-1}\\[.75ex]
x_0 & x_1 & \dots &x_{n-1}\end{smallmatrix};j_n\Big)=
\prod_{r=1}^{n-1}\frac{|x_r-x_0|^{4j_r}}{|z_r-z_0|^{2j_r}}\,
\Psi_\chi\Big(\begin{smallmatrix} j_1 & \dots & j_{n-1}\\[.25ex]
-\frac{z_1-z_0}{x_1-x_0} & \dots &-\frac{z_{n-1}-z_0}{x_{n-1}-x_0}\end{smallmatrix};j_n\Big).
\end{aligned}
\end{equation}
In order to prove \rf{Heckemod-SOV-repn}, it remains to compare the limit $z_0\ra\infty$ followed by 
$x_0\ra\infty$ of both sides of equation \rf{Heckemod-SOV-repn}.

\subsection{Comparison of limits}

Note, on the one hand, that we have
 \[
  \begin{aligned}
\lim_{z_0\ra\infty}&|z_0|^{2j_n}  \prod_{r=1}^{n-1}{|z_r-z_0|^{-2j_r}}
\Psi_{\chi}\Big(\begin{smallmatrix} j_1 & \dots & j_n\\[.25ex]
-\frac{z_1-z_0}{x_1-x_0} & \dots &-\frac{z_{n-1}-z_0}{x_{n-1}-x_0}\end{smallmatrix};j_n\Big)\\
&=\lim_{z_0\ra\infty}|z_0|^{-2J_0} 
\Psi_{\chi}\Big(\begin{smallmatrix} j_1 & \dots & j_{n-1}\\[.25ex]
\frac{z_0}{x_1-x_0} & \dots &\frac{z_0}{x_{n-1}-x_0}\end{smallmatrix};j_n\Big)
=
 \Psi_{\chi}\Big(\begin{smallmatrix} j_1 & \dots & j_{n-1}\\[.25ex]
\frac{1}{x_1-x_0} & \dots &\frac{1}{x_{n-1}-x_0} \end{smallmatrix};j_n\Big),\end{aligned}
 \]
 and using $\frac{x_0^2}{x_0-x_r}=x_0+x_r+\CO(x_0^{-1})$ furthermore
 \[
  \begin{aligned}
& \lim_{x_0\ra\infty}|x_0|^{-4j_n}  \prod_{r=1}^{n-1}{|x_r-x_0|^{4j_r}}
 \Psi_{\chi}\Big(\begin{smallmatrix} j_1 & \dots & j_{n-1}\\[.25ex]
\frac{1}{x_1-x_0} & \dots &\frac{1}{x_{n-1}-x_0} \end{smallmatrix};j_n\Big)\\
&= \lim_{x_0\ra\infty}\Psi_{\chi}\bigg(\begin{smallmatrix} j_1 & \dots & j_{n-1}\\[.5ex]
\frac{x_0^2}{x_0-x_1} & \dots &\frac{x_0^2}{x_0-x_{n-1}} \end{smallmatrix};j_n\bigg)=\lim_{x_0\ra\infty}\Psi_{\chi}\Big(\begin{smallmatrix} j_1 & \dots & j_{n-1}\\[.75ex]
x_0+x_1 & \dots & x_0+x_{n-1} \end{smallmatrix};j_n\Big)=\Psi_{\chi}\Big(\begin{smallmatrix} j_1 & \dots & j_{n-1}\\[.75ex]
x_1 & \dots & x_{n-1} \end{smallmatrix};j_n\Big).
\end{aligned}
 \]

In order to prove Theorem  \ref{Heckemod-SOV} one need to compare this with
the limits of the left side of \rf{Heckemod-SOV-repn}.
\begin{propn}
We have 
\begin{align}\label{z0lim}
& \lim_{z_0\ra\infty} |z_0|^{2j_n} \Sigma^{(n+1,1)}_{\chi}\Big(\begin{smallmatrix} j_0 & j_1 & \dots & j_{n-1}\\[.75ex]
x_0 & x_1 & \dots &x_{n-1}\end{smallmatrix};j_{n}\Big)
=\Theta^{(n,0)}_{\chi}\Big(\begin{smallmatrix}  j_1 & \dots & j_{n-1} & j_n\\[.75ex]
 x_1 & \dots &x_{n-1}& x_0\end{smallmatrix}\Big),\\
& \lim_{x_n\ra\infty}|{x_n}|^{-4j_n}
\Theta^{(n,0)}_\chi\Big(\begin{smallmatrix} j_1 & \dots & j_n\\[.95ex]
x_1 & \dots &x_n\end{smallmatrix}\Big)=\chi_n^{}\,
\Sigma^{(n,0)}_{{\chi}}\Big(\begin{smallmatrix} j_1 & \dots & j_{n-1}\\[.95ex]
x_1 & \dots &x_{n-1}\end{smallmatrix}; j_{n}\Big),
\label{xnlim}\end{align}
with $\chi_n^{}$ being  defined in \rf{HeckevsSoV(n+1)}, and 
\begin{align}\label{Theta(n)}
\Theta^{(n,0)}_{\chi}\Big(\begin{smallmatrix}  j_1 & \dots & j_{n-1} & j_n\\[.75ex]
 x_1 & \dots &x_{m-1} & x_n\end{smallmatrix}\Big)=\int\prod_{r=1}^{n}
\frac{d^2k_r\,e^{2\mathrm{i}\,\mathrm{Im}(k_rx_r)}}{\pi\,|k_r|^{2j_r+2}}&
\;\pi\de^{(2)}\big({\textstyle \sum_{r=1}^nk_r}\big)\\[-1ex] 
&\times\big|\theta(\bm{k})\big|^{2}\,
\prod_{k=0}^{n-3}\chi_{}^{}\big(y_k(\bm{k}),\bar{y}_k(\bm{k})\big), 
\notag\end{align}
with $y_0(\bm{k}),\dots,y_{n-3}(\bm{k})$ and $\theta(\bm{k})$, $\bm{k}=(k_1,\dots,k_{n})$, defined by 
\begin{equation}\label{y-to-k}
k_r=\theta(\bm{k})
\frac{\prod_{k=0}^{n-3}(z_r-y_k(\bm{k}))}{\prod_{s\neq r}(z_r-z_s)}, \qquad r=1,\dots,n,\qquad\theta(\bm{k})=
\sum_{r=1}^{n-1}k_r.
\end{equation}
\end{propn}
\begin{proof} 
Concerning the limit $z_0\ra\infty$ of the measure of integration in \rf{Fourier(n+1)}  note that 
\begin{equation}\label{hatrholim}
\lim_{z_0\ra\infty}z_0^{-1}
\hat{\rho}(\bm{k}) = k_0=-\theta(\bm{k}),
\end{equation}
on the support of $\de^{(2)}(k_0+\dots+k_{n-1})$.
It follows that the  integration  measure simplifies as
\begin{align*}
\lim_{z_0\ra\infty}& |z_0|^{2j_n}\int\prod_{r=0}^{n-1}\frac{d^2k_r}{\pi |k_r|^{2j_r+2}}\;e^{2\mathrm{i}\,\mathrm{Im}(k_rx_r)}
\pi\de^{2}(k_0+\dots+k_{n-1})|\hat{\rho}(\bm{k})|^{-2j_n}\\
&= \int\prod_{r=1}^{n-1}\frac{d^2k_r}{\pi |k_r|^{2j_r+2}}\;e^{2\mathrm{i}\,\mathrm{Im}(k_rx_r)}
|\theta(\bm{k})|^{-2j_n}e^{-2\mathrm{i}\,\mathrm{Im}(\theta(\bm{k})x_0)}\\
&= \int\prod_{r=1}^{n-1}\frac{d^2k_r}{\pi |k_r|^{2j_r+2}}\;e^{2\mathrm{i}\,\mathrm{Im}(k_rx_r)}
\int \frac{d^2k_n}{\pi |k_n|^{2j_n+2}}\;e^{2\mathrm{i}\,\mathrm{Im}(k_nx_0)}\;\pi\de^{2}(k_{n}+\theta(\bm{k}))\,
|\theta(\bm{k})|^{2}.
\notag\end{align*}
It remains to note that the equations defining $y_k(\bm{k})$ simplify for $z_0\ra\infty$ as
\begin{equation}\label{y-to-k-lim}
k_r=\frac{\hat{\rho}(\bm{k})}{z_r-z_0} 
\frac{\prod_{k=0}^{n-3}(z_r-y_k(\bm{k}))}{\prod_{s\neq r,0}(z_r-z_s)}=\theta(\bm{k}) 
\frac{\prod_{k=0}^{n-3}(z_r-y_k(\bm{k}))}{\prod_{s\neq r,0}(z_r-z_s)}+\CO(z_0^{-1}),
\end{equation}
in order to identify the integral defining $\Theta^{(n,0)}_{\chi}$ in \rf{Theta(n)}, proving \rf{z0lim}.

In order to verify \rf{xnlim} let us note that
the leading contributions to  the limit $x_n\to\infty$ of the left side are coming from the 
region where $k_n\ra 0$. It follows from \rf{y-to-k} that 
there must exist an index $k\in\{0,\dots,n-3\}$ such that $y_k\ra \infty$. 
The permutation symmetry of the integrand allows us to assume without loss of generality 
that
$y_{0}\ra \infty$.
We may next note that
\begin{equation}
\begin{aligned}
&k_r  = 
(-k_n){(z_r-y_{0})} \frac{\prod_{k=1}^{n-3}(z_r-y_{j})}{\prod_{r \neq s}^{n-1}(z_r-z_s)} =  
y_0k_n\frac{\prod_{k=1}^{n-3}(z_r-y_{k})}{\prod_{r \neq s}^{n-1}(z_r-z_s)} + \mathcal{O}(y^{-1}_0).
\end{aligned}
\end{equation}
However, we should recall that the equations relating $\bm{k}$ and $\bm{y}$ follow for $k_n\neq 0$ from 
\[
\sum_{r=1}^{n-1}\frac{k_n}{y-z_r}
=
\theta(\bm{k})(y-y_0)\frac{\prod_{k=1}^{n-3}(y-y_k)}{\prod_{r=1}^{n-1}(y-z_r)},
\]
implying, in particular, $\theta(\bm{k})=\sum_{r=1}^{n-1}k_r=-k_n$ when $k_n\neq 0$,
and $k_ny_0=\sum_{r=1}^{n-1}z_rk_r=\rho$ in the limit $k_n\ra 0$, $y_0\ra\infty$. 

The integrand of the integration over $k_n$ simplifies for $k_n\ra 0$, $y_0\ra\infty$ due to 
the asymptotic behaviour \rf{sing-chi},
\begin{equation}\label{opersingzn}
\chi(y_0,\bar{y}_0)= |y_0|^{-2j_n}C_+^{(n)}\big((1+\CO(|y_0|^{-1})\big)
+|y_0|^{2j_n+2}C_-^{(n)}\big(1+\CO(|y_0|^{-1})\big).
\end{equation}
By keeping in mind that $k_ny_0=\rho$ one can carry out the integral over $k_n$ using \rf{opersingzn}.
One of the resulting terms is proportional to a delta-distribution in $x_n$ which does not contribute to the limit 
$x_n\ra\infty$. The remaining term is  
 proportional to
\[
\int \frac{d^2 k_n}{\pi} \;\frac{e^{2\mathrm{i}\,\mathrm{Im}(k_nx_n)}}{|k_n|^{2j_n}}\bigg|\frac{\rho}{k_n}\bigg|^{2j_n+2},
\]
which can be computed with the help of the identity 
(\cite{GelfandShilov1964}, Appendix B1.7, equation (7)) 
\begin{equation}\label{Fourier power}
\frac{1}{\pi} \int\frac{d^2k}{|k|^{4j+2}}\;e^{2\mathrm{i}\,\mathrm{Im}(kx)}=\frac{\Ga(-2j)}{\Ga(2j+1)}|x|^{4j}.
\end{equation}
We thereby find that the the left side of \rf{xnlim}  is equal to 
\[
C_-^{(n)}\frac{\Ga(-2j_n)}{\Ga(2j_n+1)}
\int\prod_{r=1}^{n-1}\frac{d^2k_r}{\pi}\,\frac{e^{2\mathrm{i}\,\mathrm{Im}(k_rx_r)}}{|k_r|^{2j_r+2}}
\,\pi\,\de^{(2)}\big({\textstyle\sum_{r=1}^{n-1}}k_r\big)|\rho(\bm{k})|^{2(j_n+1)}
\prod_{k=1}^{n-3}\chi_{}^{}\big(y_k(\bm{k}),\bar{y}_k(\bm{k})\big), 
\]
which coincides with the right side of \rf{xnlim}. 
\end{proof}

\section{Hecke eigenvalue property}\label{Heckeop-SOV-pf}

We shall now
prove Theorem \ref{Heckeop-SOV} using Theorem \ref{Heckemod-SOV}. To this aim let
us note that the Hecke operator $\SH_{z_0}$ can be represented in terms of  the limit $j_0\ra -1$, $x_0\ra 0$ of 
the operator
\[
(\SR_{j_0}\Psi')(\bm{x}')=-\frac{2j_0+1}{\pi}\int_\BC d^2 x_0\;|x_0-x_0'|^{-4j_0-4}
{\Psi}'(\bm{x}').
\]
Equation \rf{Heckemod-SOV-repn} implies 
\begin{equation}
\chi_n^{}\,\mathbf{H}_{z_0}\Sigma^{(n,0)}_{{\chi}}=
\lim_{j_0\ra -1} \SR_{j_0}\Sigma^{(n+1,1)}_{\chi}.
\end{equation}
In order to compute the right side we will use the following proposition.
\begin{propn}\label{refl-propn}
The inverse SoV-transformation $\big(\mathsf{S}_{\bm{j}',i}^{(n+1)}\big)^{-1}$ satisfies the reflection property
\begin{align}\label{reflection}
\SR_{j_0}^{}\big(\mathsf{S}_{\bm{j}',i}^{(n+1)}\big)^{-1}=\frac{\Gamma(-2j_0)}{\Gamma(2j_0+2)}\big(\mathsf{S}_{\bm{j}'',i}^{(n+1)}\big)^{-1},
\end{align}
where $\bm{j}''=(-j_0-1,j_1,\dots,j_n)$ if $\bm{j}'=(j_0,j_1,\dots,j_n)$.
\end{propn}
 \begin{proof}{(Proposition \ref{refl-propn})} The identity \rf{Fourier power}
 implies
$
\SR_{j_0}^{}\big(\mathsf{F}_{\bm{j}',i}^{(n+1)}\big)^{-1}=\frac{\Gamma(-2j)}{\Gamma(2j+2)}\big(\mathsf{F}_{\bm{j}'',i}^{(n+1)}\big)^{-1}
$, which immediately 
yields Proposition \ref{refl-propn}. 
\end{proof}

We thereby find that 
\begin{equation}\label{Hecke-almost}
\begin{aligned}
&\chi_n^{}\big(\mathbf{H}_{z_0}\Sigma^{(n,0)}_{{\chi}}\big)\Big(\begin{smallmatrix} j_1 & \dots & j_{n-1}\\[.75ex]
x_1 & \dots &x_{n-1}\end{smallmatrix};j_n\Big)
=\lim_{j_0\ra -1} \frac{\Gamma(-2j_0)}{\Gamma(2j_0+2)}\Sigma^{(n+1,1)}_{\chi}\Big(\begin{smallmatrix} -j_0-1 & j_1 &\dots & j_{n-1}\\[.75ex]
0 & x_1 & \dots &x_{n-1}\end{smallmatrix};j_n\Big).
\end{aligned}
\end{equation}
Let us next observe that the identity
\[
\lim_{j\ra -1}\frac{\Gamma(-2j-1)}{\Gamma(2j+2)}|y|^{2j}=\pi\,\de^2(y),
\]
implies that 
\begin{equation*}
\lim_{j_0\ra -1} \frac{\Gamma(-2j_0)}{\Gamma(2j_0+2)}\Sigma^{(n+1,1)}_{\chi}\Big(\begin{smallmatrix} -j_0-1 & j_1 &\dots & j_{n-1}\\[.75ex]
0 & x_1 & \dots &x_{n-1}\end{smallmatrix};j_n\Big)=
\chi(z_0,\bar{z}_0)\,\Sigma^{(n,1)}_{\chi}\Big(\begin{smallmatrix} j_1 & \dots & j_{n-1}\\[.75ex]
x_1 & \dots &x_{n-1}\end{smallmatrix};j_n\Big).
\end{equation*}
Inserting this into \rf{Hecke-almost} yields 
\begin{equation}
\begin{aligned}
\chi_n^{}\mathbf{H}_{z_0}\Sigma^{(n,0)}_{{\chi}}=
{\chi(z_0,\bar{z}_0)}\,
\Sigma^{(n,1)}_{\chi}.
\end{aligned}
\end{equation}
Taken together, we therefore find that 
\begin{equation}\label{Q-eigen}
\mathbf{Q}_{z_0}\,\Sigma^{(n,1)}_{{\chi}}=
\mathbf{H}_{z_0}\mathbf{K}\,
\Sigma^{(n,1)}_{{\chi}}=\frac{1}{\chi_n^{}}\chi(z_0,\bar{z}_0)\;\Sigma^{(n,1)}_{\chi}.
\end{equation}
This concludes the proof of Theorem \ref{Heckeop-SOV}. \hfill $\square$.

\begin{rem} The form of the eigenvalue in \rf{Q-eigen} suggests that it is natural to consider the 
operator ${\mathbf{H}}_{\infty}:\CH_1\ra\CH_0$ satisfying 
${\mathbf{H}}_{\infty}\Sigma^{(n,1)}_{{\chi}}=\chi_n^{}\Sigma^{(n,0)}_{{\chi}}$, together with
$\mathbf{R}_{z_0}=\mathbf{H}_{z_0}{\mathbf{H}}_{\infty}$.
The operator $\mathbf{R}_{z_0}$ satisfies $\mathbf{R}_{z_0}\Sigma^{(n,1)}_{{\chi}}
=\chi(z_0,\bar{z}_0)\;\Sigma^{(n,1)}_{\chi}$.
One may identify ${\mathbf{H}}_{\infty}$ with one 
of the two operators appearing in the asymptotics of 
${\mathbf{H}}_{z_0}$ for $z_0\ra\infty$, see
\cite[Section 2.16.6]{EFK4}. \end{rem}


\bigskip

{\bf Acknowledgements.}
The question about the unitarity of the SOV-transformation was raised by D. Kazhdan in 2023 in a discussion with 
P. Etingof, B. Feigin,  E. Frenkel, A. Goncharov,  and one of the authors (J.T.). J.T. would like to thank all participants 
for these discussions. 
We'd in particular like to thank P. Etingof for pointing out reference \cite{GS} to us. We'd furthermore like to thank
P. Etingof and G. Felder for very useful comments on a previous version of the draft. 

We'd furthermore like to 
acknowledge the support by the Deutsche Forschungsgemeinschaft (DFG, German Research Foundation) under Germany's Excellence Strategy - EXC 2121 ``Quantum Universe'' - 390833306 and the Collaborative Research Center - SFB 1624 ``Higher structures, moduli spaces and integrability'' - 506632645.


\end{document}